\documentclass[11pt]{article}
\usepackage{color}
\usepackage{amsfonts}
\usepackage{amssymb}
\usepackage{amsthm}
\usepackage{amsmath}
\usepackage{graphics}
\usepackage{graphicx}
\usepackage{verbatim} 
\usepackage{bbm}

\usepackage{hyperref}

\newcommand{\set}[1]{\left\lbrace #1 \right\rbrace}
\newcommand{\defas}{\mathrel{\mathop{:}}=}   
\newcommand{\ie}{i.e.\;}  

\newcommand{\comm}[1]{}
\newcommand{\clEA}{\overline{\mathcal{E}_A}}
\newcommand{\clEqA}{\overline{\mathcal{E}_{q,A}}}

\newcommand{\csEqA}{\cs(\overline{\mathcal{E}_{q,A}})}
\newcommand{\e}{\mathrm{e}}
\renewcommand{\subset}{\subseteq}

\newcommand{\dualspace}[1]{\left(\mathbb{R}^{#1}\right)\!{}^{*}}

\DeclareMathAlphabet\mathbit
        \encodingdefault\rmdefault\bfdefault\itdefault
\renewcommand{\vec}[1]{\mathbit{#1}}

\DeclareMathOperator{\supp}{supp}
\DeclareMathOperator{\sgn}{sgn}

\DeclareMathOperator{\conv}{conv} 
\DeclareMathOperator{\cs}{cs} 
\DeclareMathOperator{\Min}{Min}

\theoremstyle{definition}

\newtheorem{Def}{Definition}

\theoremstyle{remark}
\newtheorem{Rem}[Def]{Remark}
\newtheorem{Exa}[Def]{Example}

\theoremstyle{plain}

\newtheorem{Lem}[Def]{Lemma}
\newtheorem{Prop}[Def]{Proposition}

\newtheorem{Thm}[Def]{Theorem}
\newtheorem{Cor}[Def]{Corollary}

\title{Support Sets in Exponential Families and Oriented Matroid Theory}

\author{{\small \bf Johannes Rauh} \\ \small Max Planck Institute \\ \small for Mathematics in the Sciences\\ \small rauh@mis.mpg.de
\and
 {\small \bf Thomas Kahle} \\ \small Isaac Newton Institute \\  \small for Mathematical Sciences \\ \small www.thomas-kahle.de
\and
{\small \bf Nihat Ay}\\ \small Max Planck Institute \\ \small for Mathematics in the Sciences\\
\small nay@mis.mpg.de}

\begin{document}


\maketitle
\begin{abstract}
  The closure of a discrete exponential family is described by a finite set of equations
  corresponding to the circuits of an underlying oriented matroid.  These equations are similar to
  the equations used in algebraic statistics, although they need not be polynomial in the general
  case.  This description allows for a combinatorial study of the possible support sets in the
  closure of an exponential family.  If two exponential families induce the same oriented matroid,
  then their closures have the same support sets.  Furthermore, the positive cocircuits give a
  parameterization of the closure of the exponential family.


\end{abstract}

\section{Introduction}
In this paper we study exponential families, which are well known statistical models with many nice properties.  Let
$\mathcal{E}$ be an exponential family on a finite set $\mathcal{X}$.
We are interested in the closure $\overline{\mathcal{E}}$ (with respect to the usual topology).
In particular we are interested in the set
\begin{equation}
  \set{ \supp (p) \subset \mathcal{X} : 
    p \in \overline{\mathcal{E}}}.
\end{equation}
of all possible support sets occurring in the closure $\overline{\mathcal{E}}$.

Finding the possible support sets and the boundary $\overline{\mathcal{E}}\setminus\mathcal{E}$ of an exponential family
$\mathcal{E}$ is an important problem in statistics.
An early discussion 
of the case where $\mathcal{X}$ is finite is due to O.~Barndorff-Nielsen~\cite{Barndorff78}.  In the general case, when
$\mathcal{X}$ is infinite, there are different possible notions of closure.  This was studied by I.~Csisz\'ar and
F.~Mat\'u\v{s} in a series of papers, see~\cite{CsiszarMatus05:Closures_of_exp_fam}; earlier results are due to
N.~Chentsov~\cite{Chentsov:Statistical_decision_rules}.  The problem is related to characterizing the convex core and
convex support of $\mathcal{E}$.  In the finite case, the notions of convex core and convex support agree, and the
convex support is a polytope, called the marginal polytope in the case of hierarchical models, a particular kind of
exponential families.  Knowing the convex core is important for studying properties of the generalized maximum
likelihood estimate, see~\cite{Barndorff78,CsiszarMatus08:GMLE_for_exp_fam}.

Computing the support sets of an exponential family is equivalent to determining
the face lattice of the convex support, which can lead to hard combinatorial problems.
For example, the so-called CUT-polytopes appear naturally when studying hierarchical models
(see \cite{KahleWenzelAy08} for the relation).  In \cite{dezalaurent97} it is shown that deciding whether a given point
lies in the CUT-polytope is NP-complete.  Nevertheless, a local exploration of the face lattice is possible using the
results presented here.
For instance, one of the authors discusses support sets of small cardinality in hierarchical
models~\cite{kahle08_degree}.  In the present paper we find a concise characterization of the support sets in general
exponential families with the help of oriented matroids.  Furthermore, we show how to describe the closure of an
exponential family parametrically and implicitly.  We hope that this will allow for further theoretical results in this
direction.  For an illustration how these results can be applied see~\cite{Rauh10:Finding_Maximizers}.

Although slightly hidden, the connection to oriented matroid theory is very natural.
The starting point, and another focus of the presentation, is the implicit description of exponential families for
discrete random variables inspired by so called Markov bases \cite{geigermeeksturmfels06}. It is described in Theorem
\ref{thm:implicitization-theorem}. We study the---not necessarily polynomial---equations that define the closure of
the exponential family and relate them to the oriented matroid of the sufficient statistics of the model. In the case of
a rational valued sufficient statistics, our observations reduce to the fact that the non-negative real part of a
toric variety is described by a circuit ideal. We emphasize how the proof of this fact uses arguments from oriented
matroid theory.

The oriented matroid also plays a role when one tries to parametrize the closure of an exponential family.  The usual
exponential parametrization naturally extends to a part of the boundary if the parameters are replaced by their
logarithms.  With the help of the positive cocircuits of the oriented matroid one can construct a parametrization of
the whole closure.

This paper is organized as follows. In Section \ref{sec:expfamsec} we develop a theory of implicit representations of
exponential families which is analogue to and inspired by algebraic statistics \cite{geigermeeksturmfels06}. In contrast
to the toric case we do not require the sufficient statistics to take integer values and thereby leave the realm of
commutative algebra. What remains is the theory of oriented matroids. We discuss how answers to the support set problem
look like in the language of oriented matroids and discuss examples coming from cyclic polytopes. These polytopes are
well known in combinatorial convexity for their extremal properties, as stated, for instance, in the Upper Bound
Theorem.  In~\ref{sec:param-descr-overl} we show how to obtain a surjective parametrization of the closure of an
exponential family.  In Section \ref{sec:OrMat} we discuss the basics of the theory of oriented matroids and reformulate
statements from Section \ref{sec:expfamsec} in this language, making the connection as clear as possible.


\section{Exponential families}
\label{sec:expfamsec}

We assume a finite set $\mathcal{X} \defas \set {1,\ldots,m}$ and denote $\mathcal{P}(\mathcal{X})$ the open simplex of
probability measures with full support on $\mathcal{X}$.  The closure of any set $M \subset \mathbb{R}^{\mathcal{X}}$,
in the standard topology of $\mathbb{R}^{n}$, is denoted by $\overline{M}$.  Any vector $n\in\mathbb{R}^{\mathcal{X}}$
can be decomposed into its positive and negative part, that is $n = n^{+} - n^{-}$ via $n^{+}(x) \defas \max (n(x),0)$
and $n^{-}(x) \defas \max (-n(x), 0)$.  For any two vectors $n, p \in\mathbb{R}^{\mathcal{X}}$ we define
\begin{equation}
  p^{n} \defas \prod_{x\in\mathcal{X}} p(x)^{n(x)},
\end{equation}
whenever this product is well defined (e.g.~when $n$ and $p$ are both non-negative).  Here $0^{0}=1$ by convention.

Let $q$ be a positive measure on $\mathcal{X}$ with full support, and let $A \in \mathbb{R}^{d\times m}$ be a matrix of
width $m$.  We denote $a_{x}$, $x\in\mathcal{X}$, the columns of $A$. Then we have
\begin{Def}
  \label{def:exponential-family}
  The \emph{exponential family} associated with the reference measure $q$ and the matrix $A$ is the set of probability
  measures
  \begin{equation}
    \label{eq:expfamdefmitAmatrix}
    \mathcal{E}_{q,A} \defas \set {p_{\theta} \in \mathcal{P}(\mathcal{X}) :
      p_{\theta}(x) = \frac{q(x)}{Z_{\theta}} \exp \left( \theta^{T}a_{x}\right), \theta \in
      \mathbb{R}^{d}},
  \end{equation}
  where $Z_{\theta} \defas \sum_{x\in\mathcal{X}} q(x) \exp \left( \theta^{T}a_{x}\right)$ ensures normalization.
  If $q(x)= 1$ for all $x\in\mathcal{X}$, i.e.~if $q$ is the uniform measure on $\mathcal{X}$, then the
  corresponding exponential family is abbreviated with $\mathcal{E}_{A}$.
\end{Def}
In the following we always assume that the matrix $A$ has the vector $(1,\ldots,1)$ in its row span.  This
means that there exists a dual vector $l_{1}\in \dualspace{d}$ which satisfies $l_{1}(a_{x}) = 1$ for all
$x\in\mathcal{X}$.  There is no loss of generality in this assumption as we can always add an additional row
$(1,\ldots,1)$ to~$A$ without changing the exponential family.

\begin{Rem}
  \label{Rem:rowspan}
  Under the assumption that the row span of $A$ contains the vector $(1,\dots,1)$ the exponential family depends on $A$
  only through its row span.
  Different matrices with the same row span lead to different parametrizations of the same exponential family.  In the
  following it will be convenient to fix one parametrization, hence we work with matrices $A$ instead of vector
  spaces.
\end{Rem}

The geometrical structure of the boundary of $\clEqA$ is encoded in the polytope of possible values that the map
$A\colon \overline{\mathcal{P}}(\mathcal{X}) \to \mathbb{R}^{d}, x\mapsto Ax$ takes:
\begin{Def}
  The \emph{convex support} of $\mathcal{E}_{q,A}$ is the polytope
  \begin{equation}
    \label{eq:defconvexsupport}
    \csEqA \defas \conv \set{a_{x} : x \in \mathcal{X}}.
  \end{equation}
  If $\mathcal{E}_{q,A}$ is a hierarchical model, then
  the convex support is also called a \emph{marginal polytope}.
\end{Def}

We will see later that the faces of $\csEqA$ are in a one-to-one correspondence with the different support sets
occurring in $\clEqA$.  Even more is true: The linear mapping of $A$ is called the \emph{moment map}.  Its restriction
to $\clEqA$ defines a homeomorphism $\clEqA \cong \csEqA$ which maps every probability measure $p\in\clEqA$ into the
face corresponding to its support, see for example~\cite{Barndorff78}.

The parametrization in \eqref{eq:expfamdefmitAmatrix} does not extend to the boundary.  One might try to overcome this
problem by allowing parameters to become infinite.  More elegantly one may replace the parameters $\theta_{i}$ by their
exponentials $\xi_{i}:=e^{\theta_{i}}$, introducing the \emph{monomial parametrization}, and allowing $\xi_{i}=0$.
However, the image of this parametrization depends on the matrix $A$, and not only the row space of $A$. We will discuss
this in section~\ref{sec:param-descr-overl}.

\subsection{Implicit representations of exponential families}
\label{sec:impl-repr-expon}
The problems with parametrizations are the main motivation to move on to an implicit description of the exponential
family~$\mathcal{E}_{q,A}$.
\begin{Thm}
  \label{thm:implicitization-theorem}
  A distribution $p$ is an element of the closure of $\mathcal{E}_{q,A}$ if and only if all the equations
  \begin{equation}
    \label{impli}
    p^{n^{+}} q^{n^{-}}= p^{n^{-}} q^{n^{+}}, \qquad \text{ for all }  n \in \ker A,
  \end{equation}
  hold for $p$.
\end{Thm}
\begin{Rem}
  This theorem is a direct generalization of Theorem~3.2 in~\cite{geigermeeksturmfels06}. There only the polynomial
  equations among \eqref{impli} are studied under the additional assumption that $A$ has only integer entries.
  Moreover, only the uniform reference measure is considered. It turns out that their proof generalizes without any
  major problem, and the proof of our theorem that we present below needs one step less, since we don't need to show the
  reduction to the polynomial equations. The different flavor of the results will be made more precise in
  Remark~\ref{rem:algstat} later.

  Our proof closely follows \cite{geigermeeksturmfels06}, but we want to explicitly point out how matroid-type arguments
  are used, the first example being Lemma~\ref{lem:polytope-lemma}.
\end{Rem}

We first state a couple of auxiliary results which are of independent interest.  The matrix $A$ and derived objects are
fixed for the rest of the considerations.  A~set $C$ is a \emph{face} of a polytope $P$ if either $C = P$ or $C$ is the
intersection of the polytope with an affine hyperplane $H$, such that all $x\in P$ with $x\notin H$ lie on one side of
the hyperplane.  Faces of maximal dimension are called \emph{facets}.  It is a fundamental result that every polytope
can equivalently be described as the convex hull of a finite set or as a finite intersection of closed linear
half-spaces (corresponding to its facets), see
\cite{ziegler94}.

In particular we are interested in the face structure of $\csEqA$.  Since we assumed that all columns of $A$ lie in the
affine hyperplane $l_{1} = 1$, we can replace every affine hyperplane $H$ by an equivalent central hyperplane (which
passes through the origin).  For the convex support $\csEqA$ we want to know which points from $\set{a_{x} : x\in
  \mathcal{X}}$ lie on each face.  This motivates the following
\begin{Def}
  \label{def:facial}
  A set $F\subset\mathcal{X}$ is called \emph{facial} if there exists a vector $c \in \mathbb{R}^{d}$ such that
  \begin{equation}
    \label{eq:facial-propery}
    c^{T} a_{y} = 0 \quad \forall y \in F, \qquad\qquad
    c^{T} a_{z} \geq 1 \quad \forall z \notin F.
  \end{equation}
\end{Def}

\begin{Lem}
  \label{lem:polytope-lemma}
  Fix 
  a subset $F\subset \mathcal{X}$. Then we have:
  \begin{itemize}
  \item 
    $F$ is facial if and only if for any $u \in \ker A$:
    \begin{equation}
      \label{eq:4}
      \supp (u^{+}) \subset F
      \Leftrightarrow \supp(u^{-}) \subset F.
    \end{equation}
  \item If $p$ is a solution to \eqref{impli}, then $\supp (p)$ is facial.
  \end{itemize}
\end{Lem}
\begin{proof}
  One direction of the first statement is direct: Let $u\in\ker A$ and suppose that $\supp(u^{+})\subset F$.  Then
  $\sum_{x\in F} u(x) a_{x} = -\sum_{x\notin F}u(x)a_{x}$, so $0 = \sum_{x\in F}u(x)c^{T} a_{x} = -\sum_{x\notin F}
  u(x)(c^{T}a_{x})$.  Since $c^{T}a_{x}>1$ and $u(x)\le 0$ for $x\notin F$ it follows that $u(x) = 0$ for $x\notin F$,
  proving one direction of the first statement.

  The opposite direction is a bit more complicated.  Here, we present a proof using elementary arguments from polytope
  theory (see, e.g., \cite{ziegler94}).  For an alternative proof using Farkas' Lemma see~\cite{geigermeeksturmfels06}.
  Assume that $F$ is not facial.  Let $F'$ be the smallest facial set containing $F$.  Let $P_{F}$ and $P_{F'}$ be the
  convex hulls of $\{a_{x}:x\in F\}$ and $\{a_{x}:x\in F'\}$.  Then $P_{F}$ contains a point $q$ from the relative
  interior of $P_{F'}$.  Therefore $q$ can be represented as $q = \sum_{x\in F}\alpha(x)a_{x} = \sum_{x\in
    F'}\beta(x)a_{x}$, where $\alpha(x)\ge 0$ for $x\in F$ and $\beta(x)>0$ for $x\in F'$.  Hence
  $u(x)\defas\alpha(x)-\beta(x)$ (where $\alpha(x)\defas 0$ for $x\notin F$ and $\beta(x)\defas 0$ for $x\notin F'$)
  defines a vector $u\in\ker A$ such that $\supp(u^{+})\subseteq F$ and $\supp(u^{-})\cap(\mathcal{X}\setminus
  F)=F'\setminus F\neq\emptyset$.

   The second statement now follows immediately: If $p$ satisfies~\eqref{impli} for some $u\in\ker A$, then the left
   hand side of~\eqref{impli} vanishes if and only if the right hand side vanishes, and by the first statement this
   implies that $\supp(p)$ is facial.
\end{proof}
Now we are ready for the proof of Theorem \ref{thm:implicitization-theorem}.
\begin{proof} [Proof of Theorem \ref{thm:implicitization-theorem}]
  The first thing to note is that it is enough to prove the theorem when $q(x) = 1$ for all $x$.  To see this note that
  $p\in\clEA$ if and only if $\lambda q p \in\clEqA$, where $\lambda>0$ is a normalizing constant, which does not appear
  in equations \eqref{impli} since they are homogeneous.

  Let $Z_{A}$ be the set of solutions of \eqref{impli}.  We first show that $\mathcal{E}_{A}$ satisfies the equations
  defining $Z_{A}$. Let $p \in \mathcal{E}_{A}$, using the parametrization we can write $p(x) = \e^{\theta^{T}a_{x}}$,
  for some vector of parameters $\theta \in \mathbb{R}^{d}$ and $a_{x}$ the $x$-column of~$A$.  Then, for each $u,v\in
  \mathbb{R}^{m}$ with~$Au=Av$, we find
  \begin{equation}
    \label{eq:1}
      p^{u} = \prod_{x\in\mathcal{X}} p(x)^{u(x)} 
       = \prod_{x\in\mathcal{X}} \left(\e^{\theta^{T}a_{x}}\right)^{u(x)} 
       = \e^{\theta^{T} Au} = \e^{\theta^{T} Av} = p^{v}.
  \end{equation}
  Thus $\mathcal{E}_{A} \subset Z_{A}$, and also $\clEA \subset
  \overline{Z}_{A} = Z_{A}$.

  Next, let $p \in Z_{A} \setminus \mathcal{E}_{A}$ and put $F:=\supp(p)$.  We construct a sequence $p_{(\mu)}$ in
  $\mathcal{E}_{A}$ that converges to $p$ as $\mu\to -\infty$.  We claim that the system of equations
  \begin{equation}
    \label{eq:logprob-equations}
    b^{T}a_{x} = \log p(x) \quad \text{ for all $x\in F$}.
  \end{equation}
  in the variables $b=(b_{1},\ldots,b_{d})$ has a solution.  Otherwise we can find numbers $v(x)$, $x\in \mathcal{X}$, such
  that $\sum_{x\in F}v(x)\log p(x)\neq 0$ and $\sum_{x\in\mathcal{X}}v(x) a_{x} = 0$.  This leads to the contradiction
  $p^{v^{+}}\neq p^{v^{-}}$.

  Fix a vector $c \in \mathbb{R}^{d}$ with property \eqref{eq:facial-propery}.  For any $\mu \in \mathbb{R}$ define
  \begin{equation*}
    p_{(\mu)} \defas p_{\mu c + b} =
    \frac{1}{Z_{\mu c + b}}
    \left(
      \e^{\mu c^{T} a_{1}} \e^{b^{T}a_{1}},\ldots, \e^{\mu c^{T}a_{m}} \e^{b^{T}a_{m}}
    \right) \in \mathcal{E}_{A}.
  \end{equation*}
  By~\eqref{eq:facial-propery} and~\eqref{eq:logprob-equations} it follows that $\lim_{\mu\to -\infty} p_{(\mu)} = p$.
  This proves the theorem.
\end{proof}

We now see that the last statement of Lemma \ref{lem:polytope-lemma} can be generalized (cf.\
\cite[Lemma~A.2]{geigermeeksturmfels06}):
\begin{Prop}
  \label{prop:facialsets}
  The following are equivalent for any set $F \subset \mathcal{X}$:
  \begin{enumerate}
  \item \label{item:1} $F$ is facial.
  \item \label{item:2} The truncation of $q$ to $F$,
    \begin{equation*}
      q_{F}(x) :=
      \begin{cases}
        \frac{q(x)}{\sum_{x\in F} q(x)}, &\text{ if }x\in F,\\
        0 & \text{ else,}
      \end{cases}
    \end{equation*}
    lies in $\clEqA$.
  \item \label{item:3} There is a vector with support $F$ in $\clEqA$.
  \end{enumerate} 
\end{Prop}

According to Theorem \ref{thm:implicitization-theorem}, in order to test whether $p$ is an element of the closure of
${\mathcal E}_{q,A}$, we have to test all the equations (\ref{impli}). The next theorem shows that it is actually enough
to check finitely many equations. For this, we need the following notion from matroid theory:

\begin{Def}
  A \emph{circuit vector} of a matrix $A$ is a nonzero vector $n\in\ker A\subseteq\mathbb{R}^{m}$ with inclusion minimal
  support, i.e if $n'\in \ker A$ satisfies $\supp(n')\subseteq\supp(n)$, then $n' = \lambda n$ for some
  $\lambda\in\mathbb{R}$.  Equivalently, $n$ is the vector of coefficients of a nontrivial linear relation $\sum_{x}
  n(x)a_{x}=0$ of the columns of $A$ with inclusion minimal support.

  A \emph{circuit} is the support set of a circuit vector.  A \emph{circuit basis} is a subset of $\ker A$ containing
  precisely one circuit vector for every circuit.\footnote{It is easy to see that a circuit basis of $\ker A$ spans
    $\ker A$.  However, in general the circuit vectors are not linearly independent.}
\end{Def}
The minimality condition implies that the circuit determines its corresponding circuit vectors up to a multiple.  If we
replace $n$ by a nonzero multiple of $n$ then equation \eqref{impli} is replaced by an equation which is equivalent over
the non-negative reals.  This means that all systems of equations corresponding to any circuit basis $C$ are equivalent.
\begin{Thm}
  \label{Theorem1}
  Let $\mathcal{E}_{q,A}$ be an exponential family.  Then $\clEqA$ equals the set of all probability distributions that satisfy
  \begin{equation}
    \label{eq:Theorem1}
    p^{c^{+}} q^{c^{-}} = p^{c^{-}} q^{c^{+}}  \text{ for all } c\in C, 
  \end{equation}
  where $C$ is a circuit basis of $A$.
\end{Thm}
The proof is based on the following two well-known lemmas.  For detailed proofs
see~\cite{BjörnerLasVergnasSturmfelsWhiteZiegler93}.  For convenience we sketch the proofs:
\begin{Lem}
  \label{Lem:circuitinvector}
  For every vector $n\in\ker A$ there exists a sign-consistent circuit vector $c\in\ker A$, i.e.~if $c(x)\neq 0\neq
  n(x)$ then $\sgn c(x)=\sgn n(x)$ for all $x\in\mathcal{X}$.
\end{Lem}
\begin{proof}
  Let $c$ be a vector with inclusion-minimal support that is sign-consistent with $n$ and satisfies
  $\supp(c)\subseteq\supp(n)$.  If $c$ is not a circuit vector, then there exists a circuit vector $c'$ with
  $\supp(c')\subset\supp(c)$.  A suitable linear combination $c + \alpha c'$, $\alpha\in\mathbb{R}$ gives a
  contradiction to the minimality of $c$.
\end{proof}
\begin{Lem}
  \label{Lem:conformalcomposition}
  Every vector $n\in\ker A$ is a finite sign-consistent sum of circuit vectors $n = \sum_{i=1}^{r} c_{i}$,
  i.e.~if $c_{i}(x)\neq 0$ then $\sgn c_{i}(x)= \sgn n(x)$ for all $x\in\mathcal{X}$.
\end{Lem}
\begin{proof}
  Use induction on the size of $\supp(n)$.  In the induction step, use a sign-consistent circuit vector, as in the last
  lemma, to reduce the support.
\end{proof}
\begin{proof}[Proof of Theorem \ref{Theorem1}]
  Again, we can assume that $q(x)=1$ for all $x\in\mathcal{X}$.
  By Theorem \ref{thm:implicitization-theorem} it suffices to show: If $p\in\mathbb{R}^{\mathcal{X}}$ satisfies
  \eqref{eq:Theorem1}, then it also satisfies $p^{n^{+}} = p^{n^{-}}$ for all $n\in\ker A$.  Write $n = \sum_{i=1}^{r}
  c_{i}$ as a sign-consistent sum of circuit vectors $c_{i}$, as in the last lemma.  Without loss of generality we can assume
  $c_{i}\in C$ for all $i$.  Then $n^{+} = \sum_{i=1}^{r} c_{i}^{+}$ and $n^{-} = \sum_{i=1}^{r} c_{i}^{-}$.  Hence $p$
  satisfies
  \begin{equation}
    p^{n^{+}} - p^{n^{-}} = p^{\sum_{i=2}^{r}c_{i}^{+}} \left(p^{c_{1}^{+}} - p^{c_{1}^{-}}\right)
                     + \left(p^{\sum_{i=2}^{r}c_{i}^{+}}- p^{\sum_{i=2}^{r}c_{i}^{-}}\right) p^{c_{1}^{-}},
  \end{equation}
  so the theorem follows easily by induction.
\end{proof}

\begin{Exa}
  \label{exa:1Dexpfam}
  Consider the following sufficient statistics:
  \begin{equation}
    A =
    \begin{pmatrix}
      1 & 1 & 1 & 1 \\
      -\alpha & 1 & 0 & 0
    \end{pmatrix},
  \end{equation}
  where $\alpha \notin \{0,1\}$ is arbitrary.  The kernel is then spanned by
  \begin{equation}
    \label{eq:spanningset}
    v_{1}= (   1 , \alpha , -1 , -\alpha )^{T} \text{ and } v_{2}= ( 1 , \alpha , -\alpha , -1 )^{T},
  \end{equation}
  but these two vectors do not form a circuit basis: They correspond to the two relations
  \begin{equation}
    \label{eq:p1p2p3p4}
    p(1)p(2)^{\alpha} = p(3) p(4)^{\alpha}
    \;\text{ and }\;
    p(1)p(2)^{\alpha} = p(3)^{\alpha} p(4).
  \end{equation}
  It follows immediately that
  \begin{equation}
    \label{eq:p3p4alpha}
    p(3)p(4)^{\alpha} = p(3)^{\alpha} p(4).
  \end{equation}
  If $p(3)p(4)$ is not zero, then we conclude $p(3)=p(4)$. However, on the boundary this does not follow from equations \eqref{eq:p1p2p3p4}:
  Possible solutions to these equations are given by
  \begin{equation}
    p_{a} = ( 0 , a , 0 , 1 - a )
    \text{  for } 0 \le a < 1.
  \end{equation}
  However, $p_{a}$ does not lie in the closure of the exponential family $\clEA$, since all members of
  $\mathcal{E}_{A}$ do satisfy $p(3)=p(4)$.

  A circuit basis of $A$ is given by the following vectors:
  \begin{subequations}
    \label{eq:Markoveqs}
    \begin{align}
      \label{eq:Markoveqs1}
      &
       ( 0 , 0 , 1 , -1 )^{T}
      & p(3)&=p(4),
      \\
      \label{eq:Markoveqs2}
      &
       ( 1 , \alpha , 0 , -1 - \alpha )^{T}
      & p(1)p(2)^{\alpha} &= p(4)^{1+\alpha},
      \\
      \label{eq:Markoveqs3}
      &
       ( 1 , \alpha , -1 - \alpha , 0 )^{T}
      & p(1)p(2)^{\alpha} &= p(3)^{1+\alpha}.
    \end{align}
  \end{subequations}
  By Theorem~\ref{Theorem1} these three equations characterize $\clEA$.
\end{Exa}

\begin{Rem}[Relation to algebraic statistics]
  \label{rem:algstat}
  In the case where the vector space $\ker A$ has a basis with integer components (for example, if $A$ is an integer
  matrix), every circuit vector is proportional to an integer circuit vector.  In this case the corresponding
  equations~\eqref{impli} are polynomial, and the theorem implies that $\clEqA$ is the non-negative real part of a
  \emph{projective variety}, i.e.~the solution set of homogeneous polynomials (see \cite{CoxLittleOShea08} for an
  introduction to commutative algebra and algebraic geometry).  If we want to use the tools of commutative algebra,
  then circuit vectors are not the right objects to consider: For example, proportional circuit vectors only yield
  equivalent equations over the non-negative reals, but we may obtain a different solution set if we allow negative real
  solutions or complex solutions.  Therefore, different integer circuit bases do not yield equivalent equations over
  $\mathbb{C}$.  This may greatly increase the running time of many algorithms of computational commutative algebra.
  One way out is to look at the \emph{ideal} $I_{c}$ generated by the binomials of all integer valued circuit vectors.
  Equivalently, $I_{c}$ corresponds to a ``prime circuit basis'' $C$ such that the components of any $n\in C$ are
  integers with greatest common divisor one.  $I_{c}$ is called the \emph{circuit ideal}, and the following discussion
  shows that the variety $V_{c}$ of $I_{c}$ equals the \emph{Zariski closure} of $\mathcal{E}_{q,A}$, i.e.~the smallest
  variety containing $\mathcal{E}_{q,A}$.

  Using $I_{c}$ is still not the best solution, since in general the circuit ideal is not \emph{radical}.  This means that there
  are polynomials that vanish on $V_{c}$ but which do not lie in $I_{c}$.  This may also increase the running time of
  algebraic algorithms.  By~\cite[Proposition 8.7]{eisenbud96:_binom_ideal} the radical of $I_{c}$ is the \emph{toric
    ideal} $I_{t}$ generated by the binomials corresponding to all $u\in\ker_{\mathbb{Z}}A$.  The fact that $I_{t}$ is
  the ideal of the Zariski closure of $\mathcal{E}_{q,A}$ was first noted in~\cite{geigermeeksturmfels06}.  It follows
  from Theorem~\ref{Theorem1}, knowing that $I_{t}$ is prime (see~\cite{eisenbud96:_binom_ideal}) and contains $I_{c}$.
  For further results on the relation between toric ideals and circuit ideals we refer to \cite{bogart07}.

  Hence, if we want to use algebraic tools, it is best to work with a \emph{Markov basis}, which can be defined as a
  finite subset of $\ker_{\mathbb{Z}}A$ such that the corresponding binomials generate $I_{t}$.  One major application
  of Markov bases makes use of the following fact, which was first noted and applied in statistics by P.~Diaconis and
  B.~Sturmfels~\cite{diaconissturmfels98}: \emph{A finite set $\mathcal{B}\subset\ker A$ is a Markov basis if and only
    if the following holds: For all $h,h'\in\mathbb{N}_{0}^{\mathcal{X}}$ such that $Ah = Ah'$ there exists a sequence
    $(b_{i})_{i=1}^{s}\subset\pm\mathcal{B}$ such that $h = h + \sum_{i=1}^{s}b_{i}$ and such that $h +
    \sum_{i=1}^{r}b_{i}\in\mathbb{N}_{0}^{\mathcal{X}}$ for all $1\le r\le s$.}
  Therefore, Markov bases elements can be used as moves to explore integer points in polytopes using a Markov Chain
  Monte Carlo algorithm.  A circuit basis is a natural generalisation, in the following sense: \emph{Let $h,h'$ be two
    non-negative vectors in $\mathbb{R}^{\mathcal{X}}$ such that $Ah = Ah'$, and let $\mathcal{C}$ be a circuit basis.
    Then there exists a sequence $(c_{i})_{i=1}^{s}\subset\mathcal{C}$ and real numbers $\alpha_{i}\in\mathbb{R}$ such
    that $h = h + \sum_{i=1}^{s}\alpha_{i}c_{i}$ and such that $h + \sum_{i=1}^{r}\alpha_{i}c_{i}$ is non-negative for
    all $1\le r\le s$} (this follows from Lemma~\ref{Lem:conformalcomposition}).  Therefore, in principle, a circuit
  basis could be used to explore polytopes using a Markov Chain Monte Carlo algorithm, drawing the elements
  $c_{i}\in\mathcal{C}$ and the coefficients $\alpha_{i}$ randomly.  Note that the converse of the statement on circuit
  bases does not hold in general.  This is related to the fact that circuit bases do not correspond to minimal systems
  of implicit equations characterizing the exponential family.  Another viewpoint is that circuit bases are more related
  to Graver bases, see~\cite{Hemmecke03:Positive_Sum_Property_and_Graver_test_sets}.

  Finding a Markov basis or a circuit basis is in general a non-trivial task.  \cite{Malkin:thesis} and
  \cite{hemmecke09:_comput} discuss algorithms for both tasks, which are implemented in the open source software package
  \textrm{4ti2}~\cite{4ti2}.  Markov basis computations tend to depend on the size of the entries of the matrix: If $A$
  has only small entries, then one may hope that there are ``enough'' vectors in $\ker_{\mathbb{Z}}A$ with small
  entries, corresponding to polynomials of low degree.  The Markov bases algorithm is related to Buchberger's algorithm,
  whose speed depends on the degrees of the starting polynomials.  Circuit computations do not depend essentially on the
  size of the entries of $A$, but the number of circuits tends to be much larger than the number of Markov basis
  elements. 
  In our experience Markov basis computations are faster when $A$ has only ``small'' entries (which is the most
  important case for applications), and circuit computations are faster when $A$ has ``large'' entries.
\end{Rem}

\begin{Exa}
  \label{exa:Markov-non-circuit}
  For an example of a Markov basis which contains noncircuits consider the matrix
    \begin{equation}
      \label{eq:non-circuit-markov-move-A-matrix}
    A =
    \begin{pmatrix}
      1 & 1 & 1 & 1 \\
      0 & 1 & 2 & 3 
    \end{pmatrix}.
  \end{equation}
  A quick calculation with the software \textrm{4ti2} gives the following circuit basis of $A$:
  \begin{equation}
    \begin{aligned}
      \label{eq:circuits-in-example}
      u_{1} := (0,  1, -2,  1),\qquad & u_{2} := (1, -2,  1,  0), \\ 
      u_{3} := (1,  0, -3,  2),\qquad & u_{4} := (2, -3,  0,  1).
    \end{aligned}
  \end{equation}
  However, any Markov basis contains $v := (1,-1,-1,1)$, which is obviously not a circuit vector.  $v$ corresponds to
  the binomial $p_{1}p_{4}-p_{2}p_{3}$.  The computation
  \begin{equation*}
    (p_{1}p_{4}-p_{2}p_{3})^{2} 
    = p_{4} (p_{1}^{2} p_{4} - p_{2}^{3}) + p_{2}^{2}(p_{3}^{2} - p_{2}p_{4})
    - 2 p_{2}p_{4} (p_{1}p_{3} - p_{2}^{2})
  \end{equation*}
  illustrates that the radical of the circuit ideal is the toric ideal, i.e.~any solution of the toric ideal also solves
  $p_{1}p_{4}-p_{2}p_{3}=0$.


  It is not easy to find a \emph{hierarchical} model whose Markov basis does not consist of circuit vectors.  In
  \cite{aoki03:_markov}, S.~Aoki and A.~Takemura give an example.  Interestingly, neither the full Markov basis nor a
  circuit basis of this model are known.
 \end{Exa}

 \begin{Rem}
  \label{rem:enumcircuits}
  Using arguments from matroid theory the number of circuits can be shown to be less than or equal to $\binom{m}{r+2}$,
  where $m=|\mathcal{X}|$ is the size of the state space and $r$ is the dimension of
  $\mathcal{E}_{q,A}$, see \cite{DosaSzalkaiLaflamme04}.  This gives an upper bound on the number of implicit equations
  describing $\clEqA$.  Note that $\binom{m}{r+2}$ is usually much larger than the codimension $m - r - 1$ of
  $\mathcal{E}_{q,A}$ in the probability simplex.  In contrast, if we only want to find an implicit description of all
  probability distributions of $\mathcal{E}_{q,A}$, which have full support, then $m - r - 1$ equations are enough: We
  can test $p\in\mathcal{E}_{q,A}$ by checking whether $\log(p/q)$ lies in the row span of $A$.  This amounts to
  checking whether $\log(p/q)$ is orthogonal to $\ker A$, which is equivalent to $m - r - 1$ equations after choosing a
  basis of $\ker A$.

  It turns out that even in the boundary the number of equations can be further reduced: In general we do not need all
  circuits for the implicit description of $\clEqA$.  For instance, in Example~\ref{exa:1Dexpfam}, the
  equations~\eqref{eq:Markoveqs2} and~\eqref{eq:Markoveqs3} are equivalent given~\eqref{eq:Markoveqs1}, i.e.~we only
  need two of the three circuits to describe $\clEqA$.  Unfortunately we do not know how to find a minimal subset of
  circuits that characterizes the closure of the exponential family.  In the algebraic case discussed in the
  previous remark this question is equivalent to determining a minimal generating set of the circuit ideal among the
  circuit vectors.
\end{Rem}

\subsection{Support sets of exponential families}
\label{sec:supp-sets-expon}
Now we focus on the following problem: Given a set $S \subseteq {\mathcal X}$, is there a probability distribution $p
\in \clEqA$ satisfying $\supp(p) = S$?
In other words, we want to characterize the set
\begin{equation}
  \mathcal{S}_{q,A} := \set{ \supp(p) : p\in\clEqA } \subseteq 2^{\mathcal{X}}.
\end{equation}

Proposition~\ref{prop:facialsets} and Lemma~\ref{lem:polytope-lemma} give the following characterization: A nonempty set
$S\subset\mathcal{X}$ is the support set of some distribution $p\in\clEqA$ if and only if the following holds for all
circuit vectors $n \in \ker A$:
\begin{itemize}
\item  $\supp(n^+) \subset S$ if and only if $\supp(n^-) \subseteq S$.
\end{itemize}
Obviously, this condition does not depend on the circuits themselves, but only on the
supports of their positive and negative part.

\begin{Def}
  A \emph{signed subset} $(M,N)$ of~$\mathcal{X}$ is a pair of disjoint subsets $M,N\subseteq\mathcal{X}$.
  Alternatively, a signed subset~$(M,N)$ can be represented as a sign vector $X\in\{-1,0,+1\}^{\mathcal{X}}$, where
  \begin{equation}
    X(x) =
    \begin{cases}
      +1, &\text{ if } x\in M,\\
      -1, &\text{ if } x\in N,\\
      0, &\text{ else.}\\
    \end{cases}
  \end{equation}
  As a slight abuse of notation, we don't make a difference between these two representations in the following.

  Consider the map
  \begin{align*}
    \sgn\colon n \mapsto (\supp(n^{+}),\supp(n^{-})),
  \end{align*}
  which associates to each vector a signed subset of~$\mathcal{X}$.
  In the sign vector representation, $\sgn$ corresponds to the usual signum mapping applied componentwise to vectors.
  The signed subset $\sgn(c)$ corresponding to a circuit vector $c\in\ker A$ shall be called a \emph{signed circuit}.
  The set of all signed circuits is denoted by
  \begin{equation}
    \mathcal{C}(A) := \pm \sgn(C) = \{ \sgn(c): c \in C \text{ or } c\in -C\},
  \end{equation}
  where $C$ is a circuit basis of $A$.
\end{Def}
Note that the set of signed circuits is twice as large as a circuit basis, so the set of signed circuits carries a lot
of redundant information, which should be removed when doing calculations with $\mathcal{C}(A)$.  However, for
theoretical purposes it is advantageous to work with the symmetric set $\mathcal{C}(A)$.

We immediately have the following
\begin{Thm}
  \label{test}
  Let $S$ be a nonempty subset of $\mathcal{X}$. Then $S \in \mathcal{S}_{q,A}$ if and only if the following holds for
  all signed circuits $(M,N) \in \mathcal{C}(A)$:
  \begin{equation}
    \label{cond}
    M \subseteq S \quad \Leftrightarrow \quad N \subseteq S.
  \end{equation}
\end{Thm}

\begin{Cor}
  \label{Cor:OrMatdeterminessupports}
  If two matrices $A_{1}$, $A_{2}$ satisfy $\mathcal{C}(A_{1}) = \mathcal{C}(A_2)$ then the possible support sets of the
  corresponding exponential families ${\mathcal E}_{q_{1},A_{1}}$ and ${\mathcal E}_{q_{2},A_{2}}$ coincide.
\end{Cor}

According to Remark \ref{rem:enumcircuits}, Theorem \ref{test} gives up to $\binom{m}{r+2}$ conditions on the
support.  Usually, some of these conditions are redundant, but it is not easy to see a priori, which conditions are
essential. %
Of course, a necessary condition for a subset $S$ of ${\mathcal X}$ to be a support set of a distribution contained in
$\clEA$ is condition \eqref{cond} restricted to pairs from a subset $\mathcal{H}\subset \mathcal{C}(A)$.  For example,
one can take ${\mathcal H} := \sgn(B)$, where $B$ is a finite subset of $\ker A$, such as a basis.

\begin{Exa}
  \label{exa:restbasis}
  We continue Example \ref{exa:1Dexpfam}.  If $\alpha>0$ then we deduce the following implications from the circuits:
  \begin{subequations}
    \begin{align}
      p(3)\neq 0 &\quad\Longleftrightarrow\quad p(4)\neq 0,
      \\
      p(1)\neq 0 \;\text{ and }\; p(2)\neq 0 &\quad\Longleftrightarrow\quad p(4)\neq 0,
      \\
      p(1)\neq 0 \;\text{ and }\; p(2)\neq 0 &\quad\Longleftrightarrow\quad p(3)\neq 0.
    \end{align}
  \end{subequations}
  Again, as above, the last two implications are equivalent given the first.

  From this it follows easily that the possible support sets in this example are $\{1\}$, $\{2\}$ and~$\{1,2,3,4\}$.
  From the spanning set \eqref{eq:spanningset} we only obtain the implication
  \begin{equation}
      p(1)\neq 0 \;\text{ and }\; p(2)\neq 0 \quad\Longleftrightarrow\quad p(3)\neq 0 \;\text{ and }\; p(4)\neq 0.
  \end{equation}
\end{Exa}

We conclude this section with two examples where a complete characterization of the face lattice of the convex support
and thus of the possible supports is easily achievable.
\begin{Exa}[Supports in the binary no-$n$-way interaction model]
  \label{sec:cyclic-polytope-example}
  Consider the binary hierarchical model \cite{KahleWenzelAy08} whose simplicial complex is the boundary of an $n$
  simplex. If $n=3$, this model is called the no-$3$-way interaction model, and its Markov bases have been recognized to
  be arbitrarily complicated \cite{loera06:_markov_bases_of_three_way}, so we do not expect to find an easy description
  of the signed circuits.  However, if we restrict ourselves to binary variables
  $x=(x_{i})_{i=1}^{n}\in\mathcal{X}\defas\set{0,1}^{n}$, the structure is very simple. In this case the exponential
  family is of dimension $2^{n}-2$, \ie of codimension 1 in the simplex, so $\ker A$ is one dimensional. It is spanned
  by the ``parity function''
  \begin{equation}
    \label{eq:n-1polytope-dep}
    e_{[n]}(x) \defas
    \begin{cases}
      -1 & \text{ if } \sum_{i=1}^{n} x_{i} \text{ is odd, } \\
      1 & \text {otherwise.}
    \end{cases}
  \end{equation}
  Using Theorem \ref{test} we can easily describe the face lattice of the marginal polytope (i.e.~convex
  support)~$P^{(n-1)}$: A set $\mathcal{Y}\subsetneq\set{0,1}^{n}$ is a support set if and only if it does not contain
  all configurations with even parity, or all configurations with odd parity. It follows that $P^{(n-1)}$ is
  \emph{neighborly}, \ie
  the convex hull of any $\lfloor\frac{\dim(P^{(n-1)})}{2}\rfloor = 2^{n-1}-1$ 
  vertices is a face of the polytope.  To see this, note that
  no set of cardinality less than $2^{n-1}$ can contain all configurations with even or odd parity. We can easily count
  the support sets by counting the non-faces of the corresponding marginal polytope, \ie all sets $\mathcal{Y}$ that
  contain either the configurations with even parity, or the configurations with odd parity. Let $s_{k}$ be the number
  of support sets of cardinality $k$, \ie the number of faces with $k$ vertices. It is given by:
  \begin{equation}
    \label{eq:svector-n-1-model}
    s_{k} = \binom {2^{n}}{k} - 2 \binom{2^{n-1}}{k-2^{n-1}},
  \end{equation}
  where 
  $\binom{m}{l} = 0$ if $l<0$.  Since this polytope has only one affine dependency~\eqref{eq:n-1polytope-dep} which
  includes all the vertices, we see that it is \emph{simplicial}, \ie all its faces are simplices.  It follows that
  $f_{k}$, the
  number of $k$-dimensional faces, 
  is given by $f_{k} = s_{k+1}$.

  Altogether we have determined the face lattice of the polytope, which means that we know the ``combinatorial type'' of
  the polytope.  It turns out that the face lattice of $P^{(n-1)}$ is isomorphic to the face lattice of the
  $(2^{n}-2)$-dimensional \emph{cyclic polytope} with $2^{n}$ vertices.
  %
  %
\end{Exa}
Next, we take a closer look at cyclic polytopes. Define the \emph{moment curve} in $\mathbb{R}^{d}$ by
\begin{equation}
  \label{eq:moment-curve}
  \vec{x} : \mathbb{R} \to \mathbb{R}^{d},\qquad t \mapsto
  \vec{x}(t) \defas \left( t, t^{2}, \cdots, t^{d} \right)^{T}.
\end{equation}
The \emph{$d$-dimensional cyclic polytope with $n$ vertices} is
  \begin{equation}
    C(d,n) \defas \conv \set{\vec{x}(t_{1}),\ldots,\vec{x}(t_{n})},
  \end{equation}
  the convex hull of $n>d$ distinct points ($t_{1} < t_{2} < \ldots < t_{n}$) on the moment curve.  The face lattice of
  a cyclic polytope can easily be described using \emph{Gale's evenness condition}, see \cite{ziegler94}.
  The cyclic polytope is simplicial and neighborly, \ie the convex hull of any $\lfloor\frac{d}{2}\rfloor$ vertices is a
  face of $C(n,d)$, 
  but even better, one has
\begin{Thm}[Upper Bound Theorem]
  \label{thm:upperboundthm}
  If $P$ is a $d$-dimensional polytope with $n=f_{0}$ vertices, then for every $k$ it has at most as many
  $k$-dimensional faces as the cyclic polytope $C(d,n)$:
  \begin{equation}
    \label{eq:upper-bound-eq}
    f_{k}(P) \leq f_{k}(C(d,n)), \quad k=0,\ldots,d.
  \end{equation}
  If equality holds for some $k$ with $\lfloor \frac{d}{2}\rfloor \leq k \leq d$ then $P$ is neighborly.
\end{Thm}
Theorem \ref{thm:upperboundthm} was conjectured by Motzkin in 1957 and its proof has a long and complicated history. The
final result is due to McMullen \cite{mcmullen70}.

The Upper Bound Theorem shows that the exponential families constructed above
have the largest number of support sets among all exponential families with the same dimension and the same number of
vertices.  Finally, we consider a cyclic polytope of dimension two
which gives an exponential family of smallest dimension containing all the vertices of the probability simplex.  The
construction is due to \cite{matus03:_maxim_of_infor_diver_from_expon_famil}.
\begin{Exa}
  Let $\mathcal{X}=\set{1,\ldots,m}$ and consider the matrix $A$, whose columns are the points on the 2-dimensional
  moment curve, augmented with row $(1,\ldots,1)$:
  \begin{equation}
    \label{eq:moment-matrix}
    A \defas
    \begin{pmatrix}
      1 & 1 & 1 & \ldots & 1 \\
      1 & 2 & 3 & \ldots & m \\
      1 & 4 & 9 & \ldots & m^{2}
    \end{pmatrix}.
  \end{equation}
  This matrix defines a two-dimensional exponential family. To approximate an arbitrary extreme point $\delta_{x}$ of
  the probability simplex, consider the parameter vector $\theta = (x^{2}, -2x, 1)^{T}$, giving rise to probability
  measures $p_{\beta \theta} = \frac{1}{Z}\exp(- \beta \theta^{T} A)$. Since $\theta^{T} a_{y} = (y-x)^{2}$, we get that
  $\lim_{\beta\to\infty} p_{\beta,\theta} = \delta_{x}$.
\end{Exa}
Summarizing we see that cyclic polytopes, owing to their extremal properties, have something to offer not only for
convex geometry, but also for statistics.

\subsection{Parametric description of \texorpdfstring{$\overline{\mathcal{E}_{q,A}}$}{the closure}}
\label{sec:param-descr-overl}
We have seen how the implicit description of an exponential family can be used as a tool to investigate the possible
support sets of an exponential family.  It is also possible to find a parametrization of the exponential family which
extends to the boundary.  However, in general this parametrization will not be injective.

In the following we assume that $A \in \mathbb{R}_{+}^{d\times m}$ has only non-negative real entries.  Since the
row\/span of $A$ contains the constant row $(1,\dots,1)$ it is always possible to replace an arbitrary matrix by a
non-negative matrix without changing the exponential family (see Remark~\ref{Rem:rowspan}).  We now replace the
exponential parametrization in~\eqref{eq:expfamdefmitAmatrix} by a ``monomial'' parametrization
\begin{equation}
  \label{eq:expfamdefmonomial}
  \mathbb{R}_{+}^{d} \ni\xi \mapsto \hat{p}_{\xi}(x) = \frac{q(x)}{Z_{\xi}} \prod_{j=1}^{d}\xi_{j}^{a_{j,x}}
\end{equation}
Then $\hat{p}_{\xi} = p_{\theta}$ if $\xi_{j} = \exp(\theta_{j})$.
Formula~\eqref{eq:expfamdefmonomial} also makes sense when some of the parameters are zero, as long as $Z_{\xi}\neq 0$.
In this case $p_{\xi}$ will describe an element from the closure of the exponential family.

The first question to answer is which parts of the closure can be reached for a given matrix $A$.  Let
\begin{equation}
  \mathcal{\hat E}_{q,A} \defas \{ \hat p_{\xi} : \xi_{j}\ge 0\text{ for all }j\text{, and }Z_{\xi}>0\}
\end{equation}
be the \emph{image of the monomial parametrization}. We want to characterize the support sets of $\mathcal{\hat
  E}_{q,A}$. 
\begin{Def}
  A subset $F\subseteq\mathcal{X}$ is called \emph{$A$-feasible} if for every $x\in\mathcal{X}\setminus F$ the set
  $\supp(a_{x}) = \{j \in \set{1,\dots,d} : a_{j,x}\neq 0\}$ is not contained in $\bigcup_{y\in F}\supp(a_{y})$.
\end{Def}
The following proposition is a mild generalization of \cite[Theorem 3.1]{geigermeeksturmfels06}; the proofs carry over
without difficulty.
\begin{Prop}
  Assume that $A$ has no negative entries.  A probability measure $p\in\overline{\mathcal{E}_{q,A}}$ lies in
  $\mathcal{\hat E}_{q,A}$ if and only if $\supp(p)$ is $A$-feasible.
\end{Prop}
\begin{proof}
  If $\hat p_{\xi}(x) = 0$ then $\xi_{j} = 0$ for some $j\in\supp(a_{x})$.  This implies $a_{j,y}=0$ for all
  $y\in\supp(\hat p_{\xi})$, so $\supp(a_{x})$ is not contained in $\bigcup_{y\in \supp(\hat p_{\xi})} \supp(a_{y})$, showing that
  $\supp(\hat p_{\xi})$ is $A$-feasible.

  For the other direction we may assume that $q(x)=1$ for all $x\in\mathcal{X}$; the general case then follows readily.
  Let $p\in\overline{\mathcal{E}_{q,A}}$.  We need to show that the system of equations
  \begin{equation}
    \prod_{i=1}^{d}\xi_{i}^{a_{i,x}} = p(x)\text{, for all }x\in\mathcal{X},
  \end{equation}
  has a solution.  In the proof of Theorem~\ref{thm:implicitization-theorem} it was shown that the related
  system~\eqref{eq:logprob-equations} has a solution $b$.  Let $\hat\xi_{i} = \exp(b_{i})$ for all $i$.  Note that
  $b_{i}$ (and also $\hat\xi_{i}$) is not restricted by equations~\eqref{eq:logprob-equations} if $i$ is not in
  $\bigcup_{x\in\supp(p)}\supp(a_{x})$.  Put
  \begin{equation}
    \xi_{i} =
    \begin{cases}
      \hat\xi_{i}, & \text{ if } i\in\bigcup_{x\in\supp(p)}\supp(a_{x}), \\
      0,          & \text { else.}
    \end{cases}
  \end{equation}
  Then $\hat{p}_{\xi}(x) = 0$ if and only if $p(x)=0$, by definition of $A$-feasibility, and $\hat{p}_{\xi}(x) = p(x)$
  if $p(x)>0$ by definition of $\hat\xi$.
\end{proof}

In order to parametrize the closure of an exponential family $\mathcal{E}_{q,A}$ with the help of the monomial
parametrization we need to find a matrix $A'$ such that $\mathcal{E}_{q,A} = \mathcal{E}_{q,A'}$ and such that every
possible support set is $A'$-feasible.  In order to define $A'$ we need the following notion from matroid theory:

\begin{Def}
  A \emph{cocircuit vector} of a matrix $A$ is a vector $v$ in the row span of $A$ with inclusion minimal support.
  A~\emph{cocircuit} is the support set of a cocircuit vector.  A cocircuit vector is positive if all of its components
  are non-negative.  A positive cocircuit is the support set of a positive cocircuit vector.
\end{Def}

Note the similarity to the definition of a circuit vector.  In fact, if $A^{*}$ is a matrix the rows of which span $\ker
A$, then the cocircuit vectors of $A$ equal the circuit vectors of $A^{*}$, and vice versa (see
Remark~\ref{sec:duality-remark}).  Consequently, Lemmas~\ref{Lem:circuitinvector} and~\ref{Lem:conformalcomposition}
remain valid if $\ker A$ is replaced by the row span of $A$ and if the word ``circuit'' is replaced by ``cocircuit''.

\begin{Thm}
\label{thm:boundary-param}
Let $\mathcal{E}_{q,A}$ be an exponential family.  
Let $A'$ be a matrix the rows of which contain one positive cocircuit vector for every positive cocircuit of $A$.  Then
$\mathcal{E}_{q,A}=\mathcal{E}_{q,A'}$, and the image of the monomial parametrization of $\mathcal{E}_{q,A'}$ consists
of $\overline{\mathcal{E}_{q,A'}}$.
\end{Thm}
\begin{proof}
  As above we may assume that $A$ has no negative entries.

  For the first statement we need to show that $A$ and $A'$ have the same row space (cf.~Remark~\ref{Rem:rowspan}).  By
  definition every row of $A'$ is a linear combination of rows from $A$.  For the other direction we may use
  Lemma~\ref{Lem:conformalcomposition} by the remark before the theorem.
  By assumption all entries of $A$ are non-negative, so every row of $A$ is a linear combination of positive cocircuit
  vectors.

  For the second statement it is enough to prove that every support set which occurs in $\overline{\mathcal{E}_{q,A}}$
  is $A'$-feasible.  Let $F=\supp(p)$ for some $p\in\overline{\mathcal{E}_{q,A}}$ and fix $x\in\mathcal{X}\setminus F$.
  By Proposition~\ref{prop:facialsets} $F$ is facial, so there exists $c\in\mathbb{R}^{d}$ as in
  Definition~\ref{def:facial}.  The row vector $v = c^{T}A$ is positive and lies in the row span of $A$.  Furthermore
  $v(x)>0$.  By Lemma~\ref{Lem:conformalcomposition} and the remark before the theorem there is a positive cocircuit
  vector $u$ such that $u(x) > 0$ and $\supp(u)\cap F = \emptyset$.  It follows that $\supp(a'_{x})$ is not a subset of
  $\bigcup_{y\in F}\supp(a'_{y})$, where $a'_{y}$ are the columns of $A'$.  Therefore $F$ is $A'$-feasible.
\end{proof}

It is easy to see that the matrix $A'$ is also the smallest matrix satisfying the conclusions of the theorem: Let $v$ be
a positive cocircuit vector.  Using the parametrization induced by $A'$ it is easy to see that
$\overline{\mathcal{E}_{q,A}}=\overline{\mathcal{E}_{q,A'}}$ contains a probability measure $p$ with support
$F:=\supp(p)=\mathcal{X}\setminus\supp(v)$ (set the parameter corresponding to the row of $v$ in $A'$ to zero).  Now
suppose that $p\in\hat{\mathcal{E}}_{q,A}$.  Then $F$ is $A$-feasible.  Fix $x\in\supp(v)$.  Then there is an index $i$
such that $a_{i,x}\neq 0$, but $a_{j,y}=0$ for all $y\in F$.  This means that the support of the $i$th row of $A$ is
contained in the support of $v$.  Since $v$ is a cocircuit vector, it follows that $A$ contains a row which is
proportional to $v$.

\begin{Exa}
  We continue examples~\ref{exa:1Dexpfam} and~\ref{exa:restbasis}.  If $\alpha>0$, then we have
  \begin{equation*}
    A' =
    \left(
    \begin{matrix}
      0 & 1+\alpha & \alpha & \alpha \\
      1+\alpha & 0 &   1    &    1
    \end{matrix}
    \right).
  \end{equation*}

  Quite generally, if the exponential family is onedimensional, then two parameters are enough to parametrize its
  closure.  This fact now follows directly from oriented matroid theory (see section~\ref{sec:OrMat}): A one-dimensional
  polytope has only two facets, therefore the corresponding oriented matroid has exactly two positive cocircuits.
\end{Exa}

\begin{Rem}
  Theorem~\ref{thm:boundary-param} is related to results of Katsabekis and Thoma
  \cite{KatsabekisThoma03:Toric_sets_and_orbits,KatsabekisThoma07:Parametrizations_Toric_Varieties_any_field}, who study
  the image of the monomial parametrization of a toric variety over an arbitrary field.  They show that there is a
  surjective monomial parametrization of any toric variety over an algebraically closed field.  However, there are
  toric varieties over $\mathbb{R}$ which have no surjective monomial parametrization.  For exponential families this
  problem disappears, since we are only concerned with non-negative real numbers.

  Related results were proved in \cite{Rapallo07:parametric_and_binomial_representation}.  Under the assumption that $A$
  has only integer entries it is shown that the monomial parametrization can be improved by replacing $A$ with the
  matrix $A''$ the rows of which consist of a \emph{Hilbert basis}\footnote{See
    \cite{Rapallo07:parametric_and_binomial_representation} for the definition of a Hilbert basis in the setting of
    linear programming.} of the row span of $A$.  However, it is not verified that each possible support set is
  $A''$-feasible.  Of course, this is evident by the previous theorem.

  Using the Hilbert basis has two disadvantages compared to the positive cocircuit vectors proposed here: Hilbert bases
  are only defined for integer $A$, and furthermore it is much more difficult to compute a Hilbert basis than to compute
  the cocircuits of a matrix.
\end{Rem}

\section{Relations to Oriented Matroids}
\label{sec:OrMat}

\newcommand{\omvector}{o.m.\ vector}
\newcommand{\omvectors}{o.m.\ vectors}

In this section the results from the previous section are related to the theory of oriented matroids.  The proofs in
this section are only sketched, since the main results of this work have already been proved directly.  We refer to
chapters 1 to 3 of \cite{BjörnerLasVergnasSturmfelsWhiteZiegler93} for a more detailed introduction to oriented matroids.

\begin{Def}
  Let $E$ be a finite set and $\mathcal{C}$ a non-empty collection of signed subsets of $E$ (see the previous section).
  For every signed set $X = (X^{+},X^{-})$ of $E$ we let $\underline X := X^{+}\cup X^{-}$ denote the \emph{support} of
  $X$.  Furthermore, the \emph{opposite signed set} is $-X = (X^{-}, X^{+})$.  Then the pair $(E,\mathcal{C})$ is called
  an \emph{oriented matroid} if the following conditions are satisfied:
  \begin{itemize}
  \item[\textbf{(C1)}] $\mathcal{C} = -\mathcal{C}$,\hfill (\emph{symmetry})
  \item[\textbf{(C2)}] for all $X,Y\in \mathcal{C}$, if $\underline X \subseteq \underline Y$, then $X=Y$ or
    $X=-Y$,\hfill (\emph{incomparability})
  \item[\textbf{(C3)}] for all $X,Y\in \mathcal{C}$, $X\neq -Y$, and $e\in X^{+}\cap Y^{-}$ there is a $Z\in\mathcal{C}$
    such that $Z^{+}\subseteq (X^{+}\cup Y^{+})\setminus\{e\}$ and $Z^{-}\subseteq(X^{-}\cup
    Y^{-})\setminus\{e\}$.\hfill (\emph{weak elimination})
  \end{itemize}
  In this case each element of~$\mathcal{C}$ is called a \emph{signed circuit}.
\end{Def}

Note that to every oriented matroid $(E,\mathcal{C})$ we have an associated unoriented matroid $(E,C)$, called the
\emph{underlying matroid}, where
\begin{equation}
  C = \left\{ X^{+}\cup X^{-} = \underline X : X\in\mathcal{C} \right\}
\end{equation}
is the set of \emph{circuits} of $(E,C)$.  In this way oriented matroids can be considered as ordinary matroids endowed
with an additional structure, namely a \emph{circuit orientation} which assigns two opposite signed circuits $\pm
X\in\mathcal{C}$ to every circuit $\underline X\in C$.

The most important example of an oriented matroid here is the oriented matroid of a matrix $A
\subseteq\mathbb{R}^{d\times m}$.  In this case let $E = \mathcal{X} = \{1,\dots,m\}$, and let
\begin{equation}
  \mathcal{C} = \bigl\{(\supp(n^{+}), \supp(n^{-})): n\in \ker A\text{ has inclusion minimal support}\bigr\}.
\end{equation}
An oriented matroid is called \emph{realizable} if it is induced by some matrix $A$.\footnote{Note that this definition
  depends, in fact, only on the kernel of $A$, compare Remark \ref{Rem:rowspan}.}


The only axiom which is not trivially fulfilled for this example is \textbf{(C3)}.  However, if we drop the minimality
condition and let $\mathcal{V} = \{(\supp(n^{+}), \supp(n^{-}): n\in \ker A\}$, then it is easy to see that
$\mathcal{V}$ satisfies \textbf{(C3)}.  Thus $(E,\mathcal{C})$ satisfies \textbf{(C3)} by the following proposition:
\begin{Prop}
  Let $\mathcal{V}$ be a nonempty collection of signed subsets of $E$ satisfying \textup{\textbf{(C1)}} and
  \textup{\textbf{(C3)}}.  Write $\Min(\mathcal{V})$ for the minimal elements of $\mathcal{V}$ (with respect to
  inclusion of supports).  Then
  \begin{enumerate}
  \item for any $X\in\mathcal{V}$ there is $Y\in\Min(\mathcal{V})$ such that $Y^{+}\subseteq X^{+}$ and $Y^{-}\subseteq
    X^{-}$.
  \item $\Min(\mathcal{V})$ is the set of circuits of an oriented matroid.
  \end{enumerate}
\end{Prop}
\begin{proof}
  \cite{BjörnerLasVergnasSturmfelsWhiteZiegler93}, Proposition 3.2.4.
\end{proof}

This illustrates how \textbf{(C2)} corresponds to the minimality condition.  It is possible to define oriented matroids
without this minimality condition using the following construction:

\begin{Def}
  The \emph{composition} of two signed subsets $X,Y$ of $E$ is the signed subset $X\circ Y$ with
  \begin{equation}
    (X\circ Y)^{+} := X^{+} \cup (Y^{+}\setminus X^{-}),
    \qquad
    (X\circ Y)^{-} := X^{-} \cup (Y^{-}\setminus X^{+}).
  \end{equation}
  A composition $X\circ Y$ is \emph{conformal} if $X$ and $Y$ are \emph{sign-consistent}, i.e.~$X^{+}\cap
  Y^{-}=\emptyset = X^{-}\cap Y^{+}$.

  An \emph{\omvector}{} of an oriented matroid is any composition of an arbitrary number of circuits.\footnote{In
    \cite{BjörnerLasVergnasSturmfelsWhiteZiegler93}, \omvectors{} are simply called vectors.  The name ``\omvector'' has
    been proposed by F.~Mat\'u\v{s} to avoid confusion.}  The set of \omvectors{} shall be denoted by $\mathcal{V}$.  If
  the oriented matroid comes from a matrix $A$, then $\mathcal{V}$ equals the set $\mathcal{V}$ from above.
\end{Def}
Note that composition is associative but not commutative in general.

The above proposition shows that an oriented matroid can alternatively be defined as a pair $(E,\mathcal{V})$, where
$\mathcal{V}$ is a collection of signed subsets satisfying \textbf{(C1)}, \textbf{(C3)} and
\begin{itemize}
\item[\textbf{(V0)}] $\emptyset\in\mathcal{V}$,
\item[\textbf{(V2)}]
  for all $X,Y\in\mathcal{V}$ we have $X\circ Y\in\mathcal{V}$.
\end{itemize}

Note that in the realizable case linear combinations of vectors correspond to composition of their sign vectors in the
following sense:
\begin{equation}
  \lim_{\epsilon\to 0, \epsilon>0}\sgn(n + \epsilon n') = \sgn(n)\circ \sgn(n').
\end{equation}
Now Lemmas \ref{Lem:circuitinvector} and \ref{Lem:conformalcomposition} correspond to the following two lemmas:
\theoremstyle{plain}
\newtheorem*{Lem1p}{Lemma \ref{Lem:circuitinvector}'}
\begin{Lem1p}
  For every \omvector{} $Y$ there exists a sign-consistent signed circuit $X$ such that $\underline X\subseteq\underline
  Y$.
\end{Lem1p}
\newtheorem*{Lem2p}{Lemma \ref{Lem:conformalcomposition}'}
\begin{Lem2p}
  Any \omvector{} is a conformal composition of circuits.
\end{Lem2p}

To every matrix $A$ we can associate a polytope which was called convex support in the last section.  Many properties of
this polytope can be translated into the language of oriented matroids.  This yields constructions which also make
sense, if the oriented matroid is not realizable.  In order to make this more precise, we need the notion of the dual
oriented matroid.  The general construction of the dual of an oriented matroid is beyond the scope of this work.  Here,
we only state the definition for realizable oriented matroids.

In the following we assume that the matrix $A$ has the constant vector $(1,\dots,1)$ in its rowspace.  This means
that all the column vectors $a_{x}$ lie in a hyperplane $l_{1}=1$.  In the general case, this can always be achieved by
adding another dimension.  Technically we require that the face lattice of the polytope spanned by the columns of $A$ is
combinatorially equivalent to the face lattice of the cone over the columns.  See also the remarks before Definition
\ref{def:facial}.

For every dual vector $l\in(\mathbb{R}^{d})^{*}$ let $N^{+}_{l} := \{ x\in\mathcal{X} : l(a_{x}) > 0 \}$ and $N^{-}_{l} := \{
x\in\mathcal{X} : l(a_{x}) < 0 \}$.  This way we can associate a signed subset $\sgn^{*}(l) := (N^{+}_{l}, N^{-}_{l})$
with $l$.  The signed subset $\sgn^{*}(l)$ is called a \emph{covector}.  Let $\mathcal{L}$ be the set of all
covectors.  If the signed subset $(N^{+}_{l}, N^{-}_{l})$ has minimal support (i.e.~``many'' vectors $a_{x}$ lie on the
hyperplane $l=0$), then $l$ is called a \emph{cocircuit vector}, and $\sgn^{*}(l)$ is called a \emph{signed cocircuit}.
The collection of all signed cocircuits shall be denoted by $\mathcal{C}^{*}$.

\begin{Lem}
  Let $(E,\mathcal{C})$ be an oriented matroid induced by a matrix $A$.  Then $(E,\mathcal{C}^{*})$ is an oriented
  matroid, called the \emph{dual oriented matroid}.
\end{Lem}
\begin{proof}
  See section 3.4 of \cite{BjörnerLasVergnasSturmfelsWhiteZiegler93}.
\end{proof}

Note that the faces of the polytope correspond to hyperplanes such that all vertices lie on one side of this hyperplane,
compare Definition \ref{def:facial}.  Thus the faces of the polytope are in a one-to-one relation with the positive
covectors, i.e.~the covectors $X = (X^{+},X^{-})$ such that $X^{-}=\emptyset$.  The face lattice of the polytope can be
reconstructed by partially ordering the positive covectors by inclusion of their supports; however, the relation needs
to be inverted: Covectors with small support correspond to faces which contain many vertices.  The empty face (which is
induced, for example, by the dual vector $l_{1}$ which defines the hyperplane containing all $a_{x}$) corresponds to the
covector $T := (\mathcal{X},\emptyset)$.  This correspondence of faces and positive covectors shows that the
parameterization of Theorem~\ref{thm:boundary-param} is, in fact, related to the face structure of the convex support.

We can apply these remarks to all abstract oriented matroids such that $T = (\mathcal{X},\emptyset)$ is a covector.
Such an oriented matroid is usually called \emph{acyclic}.  Thus a face of an acyclic oriented matroid is any positive
covector.  A \emph{vertex} is a maximal positive covector $X$ in $\mathcal{L}\setminus\{T\}$, i.e.~if $\underline{X}
\subset \underline{Y}$ for some positive covector $Y\in\mathcal{L}\setminus\{X\}$, then $Y = T$.


In this setting we have the following result, which clearly corresponds to the second statement of
Lemma~\ref{lem:polytope-lemma}:
\begin{Prop}[Las Vergnas]
  Let $(E,\mathcal{C})$ be an acyclic oriented matroid.  For any subset $F\subseteq E$ the following are equivalent:
  \begin{itemize}
  \item $F$ is a face of the oriented matroid.
  \item For every signed circuit $X\in\mathcal{C}$, if $X^{+}\subseteq F$ then $X^{-}\subseteq F$.
  \end{itemize}
\end{Prop}
\begin{proof}
  The proof of Proposition 9.1.2 in \cite{BjörnerLasVergnasSturmfelsWhiteZiegler93} applies (note that the statement of
  Proposition 9.1.2 includes an additional assumption which is never used in the proof).
\end{proof}
With the help of the moment map defined in the previous section, this proposition can be used to easily derive Theorem
\ref{test}: By the properties of the moment map, every face of the convex support corresponds to a possible support set
of an exponential family, and the proposition links this to the signed circuits of the corresponding oriented matroid.


Finally, Corollary \ref{Cor:OrMatdeterminessupports} can be rewritten as
\newtheorem*{Cor1p}{Corollary \ref{Cor:OrMatdeterminessupports}'}
\begin{Cor1p}
  The possible support sets of two exponential families coincide if they have the same oriented matroids.
\end{Cor1p}
Unfortunately, this correspondence is not one-to-one: Different oriented matroids can yield the same face lattice,
i.e.~combinatorially equivalent polytopes. A simple example is given by a regular and a non-regular octahedron as
described in \cite{ziegler94}. The special case has a name: an oriented matroid is \emph{rigid}, if its positive
covectors (\ie its face lattice) determine all covectors (\ie the whole oriented matroid). Still, Corollary
{\ref{Cor:OrMatdeterminessupports}'} implies that the instruments of the theory of oriented matroids should suffice to
describe the support sets of an exponential family.

\begin{Rem}[Importance of Duality]
  \label{sec:duality-remark}
  There are mainly two reasons why the theory of oriented matroids (as well as the theory of ordinary matroids) is
  considered important.  First, it yields an abstract framework which allows to describe a multitude of different
  combinatorial questions in a unified manner.  This, of course, does not in itself lead to any new theorem.  The
  second reason is that the theory provides the important tool of matroid duality.

  It turns out that the dual of a realizable matroid is again realizable: If $A$ is a matrix representing an
  oriented matroid $(E,\mathcal{C})$, then any matrix $A^{*}$ such that the rows of $A^{*}$ span the orthogonal
  complement of the row span of $A$ represents the oriented matroid $(E,\mathcal{C}^{*})$.

  To motivate the importance of this construction we sketch its implications for the case that the oriented matroid
  comes from a polytope.  In this case the duality is known under the name \emph{Gale transform} \cite[Chapter
  6]{ziegler94}.  A $d$-dimensional polytope with $N$ vertices can be represented by $N$ vectors in $\mathbb{R}^{d+1}$
  lying in a hyperplane.  These vectors form a $(d+1)\times N$-matrix $A$.  Now we can find an $(N-d-1)\times N$-matrix
  $A^{*}$ as above, so the dual matroid is represented by a configuration of $N$ vectors in $\mathbb{R}^{N-d-1}$.  This
  means that this construction allows us to obtain a lowdimensional image of a highdimensional polytope, as long as the
  number of vertices is not much larger than the dimension.  This method has been used for example in
  \cite{sturmfels88:_some_gale} in order to construct polytopes with quite unintuitive properties, leading to the
  rejection of some conjectures.  Furthermore, oriented matroid duality makes it possible to classify polytopes with
  ``few vertices'' by classifying vector configurations.

  The notion of dimension generalizes to arbitrary oriented matroids (and ordinary matroids).  In the general setting
  one usually talks about the \emph{rank} of a matroid, which is defined as the maximal cardinality of a subset
  $E\subseteq F$ such that $E$ contains no support of a signed circuit.  In this sense duality exchanges examples of
  high rank and low rank, where ``high'' and ``low'' is relative to $|E|$.
\end{Rem}

\begin{Rem}[Computing with Matroids]
  In Section~\ref{sec:expfamsec} we already recommended \textrm{4ti2} as a tool to compute the circuit vectors of a
  matrix.  Alternatively, \textrm{TOPCOM}~\cite{TOPCOM} is a software package which allows to do many common
  computations with oriented matroids.  Usually, the first step of a calculation is the extraction of the circuits from
  a matrix.  Both programs can only work with integer matrices.  Note that it is difficult to treat ``arbitrary'' (real)
  matrices on a computer.  Rounding the matrix entries to a floating number produces essentially rational matrices, and
  rational matrices can be turned into integer matrices by multiplying the matrix with the least common multiple of the
  denominators of all its entries.  However, even small changes to the matrix can change the oriented matroid
  (generically, $d$ vectors in $\mathbb{R}^{d}$ will be independent after adding small rounding errors).  In principle,
  the algorithms mentioned in Remark~\ref{rem:algstat} can also be used with floating point entries, if a robust
  criterion is available for checking when a floating number appearing in this algorithm numerically vanishes.
\end{Rem}



\subsection*{Acknowledgement}

The authors want to thank Fero Mat\'u\v{s} and Bastian Steudel for many discussions.  Nihat Ay has
been supported by the Santa Fe Institute.  Thomas Kahle has been working on this article during his
time at the Max-Planck-Institute for Mathematics in Leipzig, supported by the Volkswagen Foundation.

This paper is an expanded version of a contribution to WUPES'09, which did not include the material
of Section~\ref{sec:param-descr-overl}.

\bibliographystyle{amsalpha}

\bibliography{ras}

\end{document}